\documentclass[a4paper,11pt]{article}
\usepackage[T1]{fontenc}
\usepackage[latin1]{inputenc}
\usepackage{float}
\usepackage{amsmath,amsfonts}
\usepackage{amsthm}
\usepackage{amssymb}
\usepackage{authblk}
\usepackage[USenglish]{babel}
\usepackage{enumerate}
\usepackage{cite}
\usepackage[left=3cm, right=3cm, top=3cm, bottom=3cm]{geometry}
\usepackage{xcolor}
\usepackage[final, pdftex, pdfpagelabels, pdfstartview = {FitH}, bookmarks, plainpages = false, linktoc=all,linkcolor=black, citecolor=blue, urlcolor=blue,
filecolor=black]{hyperref}
\usepackage{cleveref}
\usepackage{caption}
\usepackage{subcaption}
\usepackage{dynkin-diagrams}
\usepackage{titlesec}
\usepackage{multicol}

\theoremstyle{plain}
\newtheorem{theorem}{Theorem}[section]
\newtheorem{lemma}[theorem]{Lemma}
\newtheorem{prop}[theorem]{Proposition}
\newtheorem{cor}[theorem]{Corollary}

\theoremstyle{definition}
\newtheorem{definition}[theorem]{Definition}

\newtheorem{remark}[theorem]{Remark}
\newtheorem{example}[theorem]{Example}
\newtheorem{notation}[theorem]{Notation}

\theoremstyle{remark}

\newcommand{\bbR}{\mathbb{R}}
\newcommand{\bbZ}{\mathbb{Z}}

\titleformat{\subsection}[runin]
       {\normalfont\bfseries}
       {\thesubsection}
       {0.5em}
       {}
       [.]

\numberwithin{equation}{section}

\title{Permutohedron's volume via Dyck paths}

\author{ Damian de la Fuente\thanks{LAMFA, Universit\'e de Picardie Jules Verne, Amiens, France} }

\date{ }

\begin{document}
\maketitle

\begin{abstract}
In a recent project, Castillo, Libedinsky, Plaza, and the author established a deep connection between the size of lower Bruhat intervals in affine Weyl groups and the volume of the permutohedron, showing that the former can be expressed as a linear combination of the latter.  

In this paper, we provide a formula for the volume of this polytope in terms of Dyck paths. Thus, we present a shorter, alternative, and enlightening proof of a previous formula given by Postnikov.
\end{abstract}

\section{Introduction}
In \cite{GeoFor2023} we derived a geometric formula giving the sizes of certain Bruhat intervals in affine Weyl groups in terms of convex geometry.
More precisely, to each dominant coweight $\lambda$ we associated a group element $\theta(\lambda)$ and a permutohedron $\mathsf{P}(\lambda)$ (or weight polytope outside type $A$), and we expressed the number of elements below $\theta(\lambda)$, in the strong Bruhat order, as a linear combination of the volumes of the faces of $\mathsf{P}(\lambda)$.
As a consequence, for a fixed affine Weyl group, a single polynomial encapsulates all the cardinalities of the lower intervals attached to those special elements.
Hence, in order to explicitly compute those cardinalities, a polynomial formula for the volume of the permutohedron is needed.
Motivated by this, we derived such a formula (for type $A$) in terms of Dyck paths.
This is the content of \Cref{thm: vol as dyck}.
It turned out that this result was already published by Postnikov \cite[Theorem 17.1]{postnikov2009permutohedra}, his version in terms of plane binary trees.
Our proof is shorter and relies solely on the classical \textit{pyramid formula} for convex polytopes, while also clarifying the obscure points in Postnikov's proof regarding the occurrence of the Catalan numbers.
Along the way, we also establish the polynomiality of the volume of $\mathsf{P}(\lambda)$ without relying on Brion's formula.

Let us illustrate our formula for the symmetric group $W=\mathfrak{S}_4$.
The group $W$ acts on $\bbR^4$ by permutation of coordinates.
The hyperplane $E$ of vectors whose coordinate sum is zero is $W$-stable.
For $\lambda\in E$, the permutohedron $\mathrm{P}_3(\lambda)\subset E$ is the convex polytope whose set of vertices is the orbit of $\lambda$ under $W$.
The index of $\mathrm{P}_3$ comes from the dimension of $E$.\footnote{In \cite{postnikov2009permutohedra}, Postnikov denotes this permutohedron by $P_{4}$.}
Consider the basis of $E$ composed of the following vectors
\begin{equation*}
\varpi_1=\tfrac{1}{4}(3,-1-1,-1), \qquad \varpi_2=\tfrac{1}{2}(1,1,-1,-1) \qquad \mbox{and} \qquad \varpi_3=\tfrac{1}{4}(1,1,1,-3).
\end{equation*}
We say that $x\in E$ is dominant if $x=x_1\varpi_1+x_2\varpi_2+x_3\varpi_3$ with $x_1,x_2,x_3\in\bbR_{\geq 0}$.
The volume (we are using the Euclidean volume, see \Cref{sec: vol conventions}) of the permutohedron $\mathcal{V}_3$ is the polynomial in  $\bbR[x_1,x_2,x_3]$ determined by
\begin{equation}
\mathcal{V}_3(x_1, x_2,x_3) = \mathrm{Vol}\left(\mathrm{P}_3(x_1\varpi_1+x_2\varpi_2+x_3\varpi_3)\right), \qquad \forall \, x_1,x_2,x_3\in\bbR_{\geq 0}.
\end{equation}
It suffices to consider only dominant vectors, as to any vector in $E$ it corresponds a unique dominant vector in its $W$-orbit.
We express $\mathcal{V}_3$ in terms of the following polynomials.
For non-negative integers $d,i,u$ with $1\leq i\leq d\leq 3$ and $u+d\leq 3$, let
\begin{equation}\label{eq: intro gamma pols}
    \Gamma_{d,i}[u](x_1,x_2,x_3)=\dfrac{1}{d}\binom{d+1}{i}\sum_{j=1}^dc_{d,i,j}\,x_{j+u},
\end{equation}
where $[-]$ denotes a shift in the variables and $c_{d,i,j}=\mathrm{min}\{i,j\}-\dfrac{ij}{d+1}$ (these coefficients are the entries of the inverse of the Cartan matrix of type $A_d$).

To each $3$-Dyck path $D$ (from $(0,0)$ to $(3,3)$), we associate a polynomial $\Gamma_D$.
First, we associate a polynomial $\Gamma_{d,i}[u]$ to each north step of $D$, and then we define $\Gamma_D$ as the product of the polynomials at each of its north steps.
In Figure \ref{fig: gammaN}, we have fixed a Dyck path $D$ and computed the polynomials of each north step in it.
Let $N\in D$ be a north step (the thick segment) and $P=(u,u')$ be its initial point.
Draw a segment along $y=x+u'-u$ starting from $P$, as long as it is contained in $D$ (the dotted line touching $N$).
Consider the set of $j>0$ such that $P+(j,j)$ is in the segment and $D$ passes trough it (the thick squares).
The polynomial associated to $N$ is $\Gamma_N=\Gamma_{d,i}[u]$, where $d,i$ are the maximal and minimal elements of this set, respectively.
\begin{figure}[H]
\centering
\captionsetup[subfigure]{labelformat=empty}
    \begin{subfigure}[b]{0.3\textwidth}
        \centering
        \includegraphics[width=3.2cm]{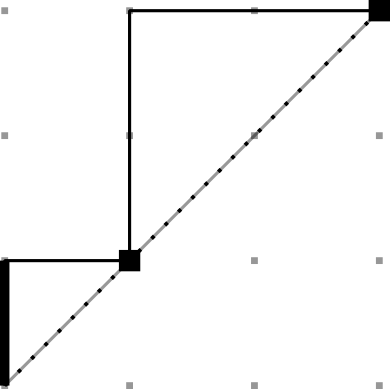}
        \caption{$\Gamma_{3,1}[0]=x_1 + \tfrac{2}{3}x_2 + \tfrac{1}{3}x_3$}
    \end{subfigure}
    \hfill
    \begin{subfigure}[b]{0.3\textwidth}
        \centering
        \includegraphics[width=3.2cm]{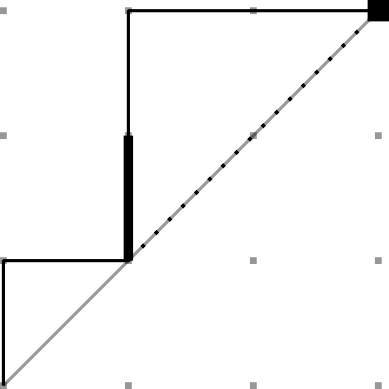}
        \caption{$\Gamma_{2,2}[1]=\tfrac{1}{2}x_2 + x_3$}
    \end{subfigure}
    \hfill
    \begin{subfigure}[b]{0.3\textwidth}
        \centering
        \includegraphics[width=3.2cm]{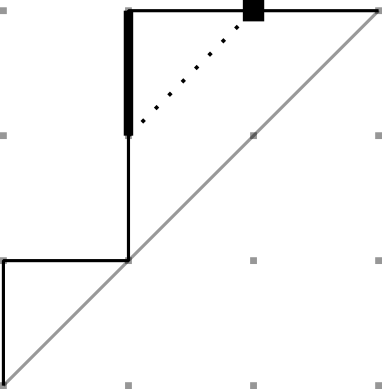}
        \caption{$\Gamma_{1,1}[1]=x_2$}
    \end{subfigure}
\caption{The polynomial of a north step}
\label{fig: gammaN}
\end{figure}
Therefore, 
\begin{equation}
\Gamma_D(x_1,x_2,x_3)=\prod_{N\in D}\Gamma_N(x_1,x_2,x_3)=(x_1 + \tfrac{2}{3}x_2 + \tfrac{1}{3}x_3)(\tfrac{1}{2}x_2 + x_3)(x_2).
\end{equation}
Following the same procedure, in \Cref{fig: gammaDyck3} we have computed the polynomials $\Gamma_D$, for all $3$-Dyck paths $D$.
Then, the volume $\mathcal{V}_3$ of the permutohedron is computed as $\sqrt{4}$ times the sum of all the polynomials $\Gamma_D$.
After factorization, we get
\begin{align*}
\mathcal{V}_3&=\mathrm{Vol}(\mathsf{P}_3(x_1,x_2,x_3))\\
&=\tfrac{1}{3}x_1^3 + 2x_1^2x_2 + 4x_1x_2^2 + \tfrac{4}{3}x_2^3 + 3x_1^2x_3 + 12x_1x_2x_3 + 4x_2^2x_3 + 3x_1x_3^2 + 2x_2x_3^2 + \tfrac{1}{3}x_3^3.
\end{align*}
In general, in \Cref{thm: vol as dyck} (more precisely in \Cref{cor: vol gamma normalized}), we show that
\begin{equation}\label{eq: vol as dyck intro}
\mathcal{V}_n(x_1,\ldots,x_n)=\sqrt{n+1}\,\sum_{D}\Gamma_D(x_1,\ldots,x_n),
\end{equation}
where the sum is over all $n$-Dyck paths.

\begin{figure}[ht]
\centering
\captionsetup[subfigure]{labelformat=empty}
    \begin{subfigure}[b]{0.3\textwidth}
        \centering
        \includegraphics[width=3.2cm]{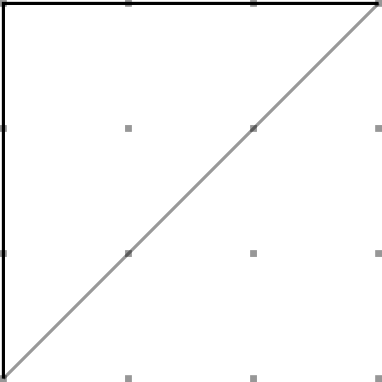}
        \caption{$(\tfrac{1}{3}x_1 + \tfrac{2}{3}x_2 + x_3)(\tfrac{1}{2}x_1 + x_2)(x_1)$}
    \end{subfigure}
    \hfill
    \begin{subfigure}[b]{0.3\textwidth}
        \centering
        \includegraphics[width=3.2cm]{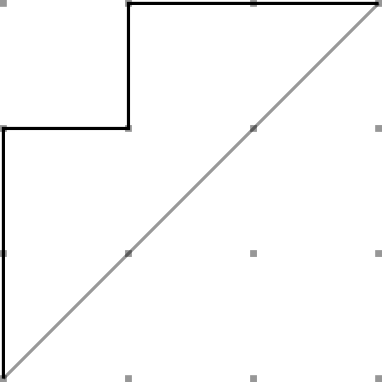}
        \caption{$(\tfrac{1}{3}x_1 + \tfrac{2}{3}x_2 + x_3)(x_1 + \tfrac{1}{2}x_2)(x_2)$}
    \end{subfigure}
    \hfill
    \begin{subfigure}[b]{0.3\textwidth}
        \centering
        \includegraphics[width=3.2cm]{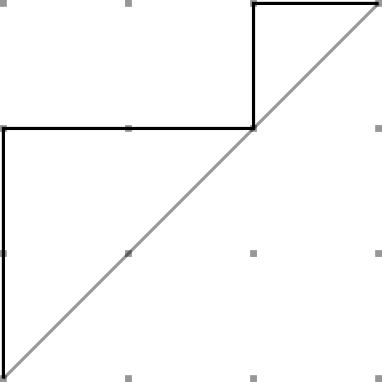}
        \caption{$(x_1 + 2x_2 + x_3)(x_1)(x_3)$}
    \end{subfigure}
    \vfill
    \vspace{2em}
    \begin{subfigure}[b]{0.3\textwidth}
        \centering
        \includegraphics[width=3.2cm]{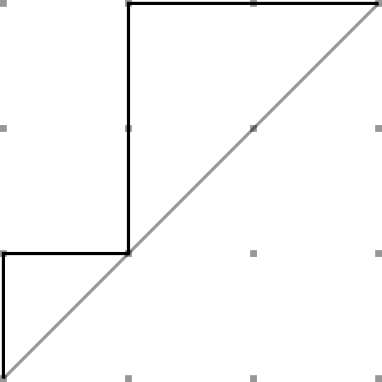}
        \caption{$(x_1 + \tfrac{2}{3}x_2 + \tfrac{1}{3}x_3)(\tfrac{1}{2}x_2 + x_3)(x_2)$}
    \end{subfigure}
    \hspace{2em}
    \begin{subfigure}[b]{0.3\textwidth}
        \centering
        \includegraphics[width=3.2cm]{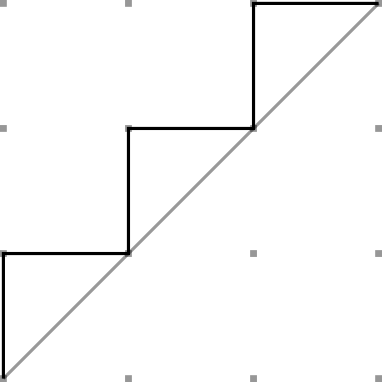}
        \caption{$(x_1 + \tfrac{2}{3}x_2 + \tfrac{1}{3}x_3)(x_2 + \tfrac{1}{2}x_3)(x_3)$}
    \end{subfigure}
\caption{The polynomial of a Dyck path}
\label{fig: gammaDyck3}
\end{figure}

\section{Preliminaries}

\subsection{The root system of type \texorpdfstring{$A_n$}{Lg}}\label{subsec: typeA}
For a more detailed discussion on root systems and Weyl groups, we refer to \cite{Bour46, humphreys1992reflection}.
Consider the Euclidean space $\bbR^{n+1}$ with canonical basis $\varepsilon_1,\ldots,\varepsilon_{n+1}$ and inner product $(-,-)$.
Let $E$ be the hyperplane of vectors whose coordinate sum is zero.
The set
\begin{equation}
    \Phi= \{  \varepsilon_i-\varepsilon_j \in E \mid 1\leq i,j \leq n+1, \, i\neq j \}
\end{equation}
is a root system of type $A_n$.
Its simple roots are the vectors $\alpha_i=\varepsilon_i - \varepsilon_{i+1}$, for $ 1\leq i \leq n$.
We denote by $\Delta\subset\Phi$ the set of simple roots.

The fundamental weights ${\varpi_1,\ldots,\varpi_n\in E}$ are defined by the equations ${(\varpi_i,\alpha_j)=\delta_{ij}}$.
We denote by $\Lambda$ the set of fundamental weights.\footnote{In the literature, $\Lambda$ usually denotes the lattice spanned by the set of fundamental weights and there is no standard symbol to denote this set.}
Both $\Delta$ and $\Lambda$ form a basis of $E$.
We define the dominant region $C^+\subset E$ to be the cone spanned by the fundamental weights
\begin{equation}
    C^+= \{ \lambda \in E \, | \, (\lambda,\alpha_i) \geq 0, \, \mbox{ for all }  1\leq i  \leq n  \},
\end{equation}
and we say that the vector $\lambda\in E$ is dominant if it lies in $C^+$.

Given a root $\alpha$, we denote by $s_\alpha$ the reflection through the hyperplane orthogonal to $\alpha$.
In formulas, $$  s_\alpha(\lambda) = \lambda - (\lambda, \alpha) \alpha. $$
Let  $S=\{s_\alpha\mid \alpha\in\Delta\}$.
The elements of $S$ are called simple reflections.
For $\alpha_i\in\Delta$, we use the notation $s_i=s_{\alpha_i}\in S$.
The subgroup $W$ of the orthogonal transformations of $E$ generated by $S$ is the  Weyl group attached to $\Phi$.
We have a group isomorphism
\begin{align*}
    W&\xlongrightarrow{\sim}\mathfrak{S}_{n+1}\\
    s_i&\mapsto (i,i+1)
\end{align*}
where $\mathfrak{S}_{n+1}$ is the symmetric group.
The pair $(W,S)$ is a Coxeter system of type $A_n$.
For $J\subset S$, we denote by $W_J$ the subgroup of $W$ generated by $J$.
The Coxeter graph of $(W,S)$ is given by
\begin{center}
$\mathsf{G}=$
\begin{dynkinDiagram}[labels={s_1,s_2,s_{n-1},s_n},edge length=1.5cm,root radius=0.09cm]{A}{}
\end{dynkinDiagram}
\end{center}
\begin{definition}\label{def: connected J}
For $J\subset S,$ we denote by $\mathsf{G}_J$ the corresponding induced subgraph of $\mathsf{G}$.
We define
\begin{equation}
\mathsf{CC}(J)=\left\{ \mathsf{Vert}(\mathrm{C})\mid \mathrm{C} \mbox{ is a connected component of } \mathsf{G}_J \right\},
\end{equation}
where $\mathsf{Vert}(\mathrm{C})$ is the vertex set of $\mathrm{C}$.
We say that $J$ is \textit{connected} whenever $\mathsf{G}_J$ is a connected graph, i.e., $\mathsf{CC}(J)=\{J\}$.
\end{definition}

The Cartan matrix $\mathsf{A}_n$ is the $n\times n$ matrix given by $(\mathsf{A}_n)_{ij}=(\alpha_i,\alpha_j)$
\begin{equation*}
\mathsf{A}_n=
    \left(\begin{array}{ccccc}
2 & -1 & 0 & \cdots & 0\\
-1 & 2 & -1 & \ddots & \vdots\\
0 & -1 & 2 & \ddots & 0\\
\vdots & \ddots & \ddots & \ddots & -1\\
0 & \cdots & 0 & -1 & 2\end{array}\right).
\end{equation*}
This matrix is invertible and its inverse is given by (see \cite{CartanInverse})
\begin{equation}\label{eq: inversa cartan}
(\mathsf{A}_n^{-1})_{ij}=(\varpi_i,\varpi_j)=\mathrm{min}\{i,j\}-\dfrac{ij}{n+1}.
\end{equation}

\begin{remark}
We have a bijection between the simple reflections $S$, the simple roots $\Delta$ and the fundamental weights $\Lambda$.
Throughout this paper we will index these four sets with $S$, so given $s\in S$, it corresponds $\alpha_s\in \Delta$ and $\varpi_s\in\Lambda$.
Likewise, for $J\subset S$ we have two corresponding sets $\Delta_J\subset\Delta$ and $\Lambda_J\subset \Lambda$.
\end{remark}

\subsection{Volumes and polytopes}\label{sec: vol conventions}
The Euclidean volume on $\bbR^{n+1}$ is the unique translation invariant measure $\mathrm{Vol}_{n+1}$ scaled so that $\mathrm{Vol}_{n+1}([0,1]^{n+1})=1$, where $[0,1]^{n+1}$ is the unit cube.
It satisfies the property $$|\mathrm{det}(v_1,\ldots,v_{n+1})| = \mathrm{Vol}_{n+1}( \Pi_{v_1,\ldots, v_{n+1}}),$$
where $\Pi_{v_1,\ldots, v_{n+1}}$ is the parallelepiped spanned by the vectors $v_1,\ldots, v_{n+1}\in\bbR^{n+1}$.

For a $k$-dimensional subspace $U$ of $\bbR^{n+1}$, the induced volume $\mathrm{Vol}_U$ is determined by the property $$\mathrm{Vol}_{U}( \Pi_{u_1,\ldots, u_k}) = |\mathrm{det}(u_1,\ldots,u_k,v_{k+1},\ldots,v_{n+1})|,$$
where $\{u_1,\ldots,u_k\}$ is a basis of $U$ and $\{v_{k+1},\ldots,v_{n+1}\}$ an orthonormal basis for the orthogonal complement of $U$ in $\bbR^{n+1}$.
We write $\mathrm{Vol}_k=\mathrm{Vol}_U$, by abuse of notation and we call it the $k$-dimensional volume.

We now recall the classical \textit{pyramid formula} for convex polytopes (see\cite[Lemma 5.1.1]{Schneider_2013}, for example).

\begin{lemma}\label{lem: piramidal politopo}
Let $O\in\mathsf{P}\subset U$ be a convex polytope $\mathsf{P}$ with an interior point $O$, sitting in a $k$-dimensional subspace $U$ of $\bbR^{n+1}$.
Then
\begin{equation}
\mathrm{Vol}_k(\mathsf{P})=\frac{1}{k}\sum_{\mathsf{F}\text{ facet of } \mathsf{P}} d(O,\mathsf{F}) \, \mathrm{Vol}_{k-1}(\mathsf{F}),
\end{equation}
where $d(O,\mathsf{F})$ is the distance from $O$ to the space spanned by $\mathsf{F}$.
\end{lemma}

The pyramid formula, as the name suggests, is derived by partitioning $\mathsf{P}$ into pyramids from $O$ to its facets.
This formula will be crucial to us when applying it to the permutohedron in \Cref{lem: piramidal}.

\subsection{The permutohedron}
We recollect some known facts about the permutohedron (see \cite{renner2009descent, postnikov2009permutohedra}).
We denote by $\mathbf{0}$ the origin of $E$.

\begin{definition}
For ${\lambda\in E}$ we define the permutohedron $\mathsf{P}_n(\lambda)\subset E$ as the convex hull of the $W$-orbit of $\lambda$
\begin{equation}
\mathsf{P}_n(\lambda)=\mathrm{Conv}(W\cdot\lambda).
\end{equation}
If $\lambda=\mathbf{0}$ then $\mathsf{P}_n(\lambda)=\{\mathbf{0}\}$.
Otherwise, the permutohedron $\mathsf{P}_n(\lambda)$ is full dimensional.
\end{definition}

We first notice that for any $\lambda\in E$, there is only one dominant vector in the orbit $W\cdot\lambda$.
Therefore, in order to describe the face structure of $\mathsf{P}_n(\lambda)$, it suffices to consider $\lambda\in C^+$.
We consider two cases.

Suppose that $\lambda$ is not stabilized by any simple reflection.
For $s\in S$, the affine hyperplane ${\lambda+\langle\Delta\setminus\{\alpha_s\}\rangle}$ is a supporting hyperplane of $\mathsf{P}_n(\lambda)$.
The intersection of this hyperplane with the permutohedron defines a facet
\begin{equation}
\mathsf{F}^s(\lambda)=\mathrm{Conv}( W_{S\setminus\{s\}}\cdot \lambda ).
\end{equation}
Now to any $J\subset S$ we associate the face $\mathsf{F}_J(\lambda)$ of $\mathsf{P}_n$, obtained as the intersection of all the facets $\mathsf{F}^s(\lambda)$ for $s\notin J$.
We can equivalently describe this face as
\begin{equation}\label{eq: face_J W_J}
\mathsf{F}_J(\lambda)=\mathrm{Conv}(W_J\cdot\lambda).
\end{equation}
It is a fact that these faces are all the faces of $\mathsf{P}_n(\lambda)$ containing $\lambda$.
Since $\lambda$ is not stabilized by any $s\in S$, one can prove that the space spanned by $\mathsf{F}_J(\lambda)$ is $|J|$-dimensional.
In this case, we say that the face is \textit{non-degenerate}.
We note that
\begin{equation*}
\mathsf{F}_J(\lambda)=\begin{cases}
    \mathsf{P}_n(\lambda) & \text{ if } J=S,\\
    \mathsf{F}^s(\lambda) & \text{ if } J={S\setminus\{s\}},\\
    \{\lambda\} & \text{ if } J=\emptyset.
\end{cases}
\end{equation*}

For general $\lambda\in C^+$, let us denote by $ Z(\lambda)\subset S$ the subset of simple reflections stabilizing $\lambda$.

\begin{prop}\label{prop: face structure}
Let $\lambda\in C^+$ with $Z(\lambda)=\emptyset$.
Then the $d$-dimensional faces of $\mathsf{P}_n(\lambda)$ are precisely the $W$-orbit of the faces $\mathsf{F}_J(\lambda)$, for all $J\subset S$ with $|J|=d$.
Furthermore, for $J\subset S$ fixed, the $W$-orbit of $\mathsf{F}_J(\lambda)$ is composed of $[W:W_J]$ faces of dimension $|J|$.
\end{prop}

On the other hand, if $\lambda$ is stabilized by $Z(\lambda)\neq\emptyset$, then the $W$-orbit of all the $\mathsf{F}_J(\lambda)$ still describes the faces of $\mathsf{P}_n(\lambda)$, however this is no longer a bijection as there are degenerated faces.

\begin{prop}\label{prop: general face structure}
Let $\lambda\in C^+$ with $Z(\lambda)\neq\emptyset$.
Then \Cref{prop: face structure} holds for all $J\subset S$ satisfying that no connected component of $\mathsf{G}_J$ is contained in $\mathsf{G}_{Z(\lambda)}$.
\end{prop}
This extra condition is empty if $Z(\lambda)=\emptyset$.

\section{The result}

\begin{definition}\label{def: vols V_J}
For $\lambda\in C^+$ and $J\subset S$ with $|J|=d$, we define
\begin{equation}
\mathcal{V}_J(\lambda) = \mathrm{Vol}_d\left( \mathsf{F}_J(\lambda) \right)=\mathrm{Vol}_d\left(\mathrm{Conv}(W_J\cdot\lambda)\right),
\end{equation}
with $ \mathsf{F}_J(\lambda)\subset\mathsf{P}_n(\lambda)$.
We write $\mathcal{V}_n=\mathcal{V}_S$ and we set $\mathcal{V}_\emptyset(\lambda)=1$.
\end{definition}

Note that $\mathcal{V}_J(\lambda)=0$ whenever the face $\mathsf{F}_J(\lambda)$ is degenerated.
Our main objective is to compute $\mathcal{V}_n(\lambda)=\mathrm{Vol}_n(\mathsf{P}_n(\lambda))$.
We recall the formula \cite[Lemma 4.5]{GeoFor2023} for $\mathcal{V}_n(\lambda)$, which we prove here for the sake of completeness.

\begin{lemma}\label{lem: piramidal} For every $\lambda\in C^+$, we have
\begin{equation}\label{eq: piramidal}
\mathcal{V}_n(\lambda)=\dfrac{1}{n}\sum_{s\in S}\left[W:W_{S\setminus\{s\}}\right]\,\dfrac{\left(\lambda,\varpi_s\right)}{\left\Vert\varpi_s\right\Vert} \ \mathcal{V}_{S\setminus\{s\}}(\lambda).
\end{equation}
\end{lemma}

\begin{proof}
Applying \Cref{lem: piramidal politopo} to $\mathbf{0}\in\mathsf{P}_n(\lambda)\subset E$, we get
\begin{align}
\begin{split}
\mathcal{V}_n(\lambda) &= \frac{1}{n}\sum_{\mathsf{F}\text{ facet of } \mathsf{P}_n(\lambda)} d(\mathbf{0},\mathsf{F}) \, \mathrm{Vol}_{k-1}(\mathsf{F})\\
&= \dfrac{1}{n}\sum_{s\in S}\left[W:W_{S\setminus\{s\}}\right]\,d(\mathbf{0},\mathsf{F}^s(\lambda)) \ \mathcal{V}_{S\setminus\{s\}}(\lambda) \\
&= \dfrac{1}{n}\sum_{s\in S}\left[W:W_{S\setminus\{s\}}\right]\,\dfrac{\left(\lambda,\varpi_s\right)}{\left\Vert\varpi_s\right\Vert} \ \mathcal{V}_{S\setminus\{s\}}(\lambda).
\end{split}
\end{align}
The second equality is derived from \Cref{prop: general face structure}.
The third equality follows from the classical formula for computing the distance from a point to a hyperplane. 
\end{proof}

We derive the formula \eqref{eq: vol as dyck intro} for the volume of the permutohedron in two steps.
We first note that the formula \eqref{eq: piramidal} is not recursive, since we cannot apply it to $J\neq S$.
Our first step is to establish basic properties of the volumes $\mathcal{V}_J$, which will turn \eqref{eq: piramidal} into a recursive formula (\Cref{prop: recurrence}).
This will also imply that these volumes can be seen as polynomials.
Since this step is fairly technical and offers limited insight, we suggest the reader skip ahead to \Cref{prop: recurrence}.
Secondly, we solve this recursive formula giving the resulting polynomial in terms of Dyck paths.

\subsection{Step 1}

For what follows we will perform computations involving varying dimensions.
In order to be precise, we will index everything by $(-)^n$.
For example, in type $A_n$, the Weyl group is $W^n$, the ambient space becomes $E^n$ and the simple roots, fundamental weights and simple reflections will be denoted by $\alpha_i^n,\varpi_i^n,s_i^n$, respectively.
Recall \Cref{def: connected J}.
It is immediate that $J\subset S^n$ is connected if and only if it is of the form
\begin{equation}\label{eq: S shift}
S_d^n[u]=\{s_{1+u}^n,s_{2+u}^n,\ldots,s_{d+u}^n\},
\end{equation}
for some $u,d$ with $u+d\leq n$.
If $u=0$ we put $S_d^n=S_d^n[0]$.
We have $S_n^n=S^n$.

Let us denote by $E^n(X)$ the subspace of $E^n$ with basis $X\subset E^n$.
For $J\subset S^n$, we define the linear map $\pi_J^n:E^n\rightarrow E^n(\Lambda_J^n)$ given by
\begin{equation}
 \pi_J^n(\varpi^n_s)=
\begin{cases}
\varpi_s^n & \mbox{if } s\in J\\
\mathbf{0} & \mbox{otherwise}
\end{cases}
\end{equation}
If $J=S_d^n[u]$, we also define the linear isomorphism $U_J^n:E^n(\Lambda^n_J)\xlongrightarrow{\sim} E^d$ given by
\begin{equation}
U_J^n(\varpi_{i}^n)=\varpi_{i-u}^d \quad \mbox{for all}\quad i\in\{1+u,\ldots,d+u\}.
\end{equation}
In the following lemma we collect some basic properties of the volumes $\mathcal{V}_J$.

\begin{lemma}\label{lem: volgralfacts}
Let $J\subset S^n$ and $\lambda\in (C^+)^n$.
The following equalities hold.
\begin{enumerate}
\item\label{lem: vol J coord} $\mathcal{V}_J(\lambda)=\mathcal{V}_J(\pi_J^n(\lambda))$.
\item\label{lem: vol cc} $\displaystyle \mathcal{V}_J(\lambda)=\prod_{K\in\mathsf{CC}(J)}\mathcal{V}_K(\lambda)$.
\item\label{lem: vol shift} If $J=S_d^n[u]$ is connected, then $\mathcal{V}_J(\lambda)=\mathcal{V}_{d}(U_J^n(\pi_J^n(\lambda)))$.
\end{enumerate}
\end{lemma}

\begin{proof}
Parts \ref{lem: vol J coord} and \ref{lem: vol cc} are straightforward and are left as an exercise to the reader.
We prove Part \ref{lem: vol shift}.
Let $T:E^n(\Delta^n_J)\xlongrightarrow{\sim} E^d$ and $\varphi:W^n_J\xlongrightarrow{\sim} W^d$ be the linear and group isomorphisms given by
\begin{equation}\label{eq: def T phi}
T(\alpha_{i}^n)=\alpha_{i-u}^d \qquad \mbox{and} \qquad \varphi(s_{i}^n)=s_{i-u}^d, \qquad \forall \, i\in\{1+u,\ldots,d+u\}.
\end{equation}
We notice that $T$ also preserves the inner product, as $(\alpha_{i}^n,\alpha_{j}^n)=(\alpha_{i-u}^d,\alpha_{j-u}^d)$ for all $i,j\in\{1+u,\ldots,d+u\}$.
Let $\nu\in  E^n(\Delta^n_J)$, $w\in W^n_J$ and $\mu\in  E^n(\Lambda^n_J)$.
We claim the following:
\raggedcolumns%
\raggedright%
\begin{multicols}{2}
\begin{enumerate}[(a)]
\item\label{claim1} $w(\nu)\in E^n(\Delta^n_J)$.
\item\label{claim2} $T(w(\nu))=\varphi(w)(T(\nu))$. \columnbreak
\item\label{claim3} $w(\mu)-\mu\in E^n(\Delta^n_J)$.
\item\label{claim4} $T(w(\mu)-\mu)=\varphi(w)(U_J^n(\mu))-U_J^n(\mu)$.
\end{enumerate}
\end{multicols}
The claims are clear for $w=\mathrm{id}$.
By linearity, it suffices to consider $\nu=\alpha_{j}^n$ and $\mu=\varpi_{j}^n$.
Let $\ell$ be the length function associated to the Coxeter system $(W^n,S^n)$, and let  $s^n_{i}\in J=S_d^n[u]$ such that $\ell(ws^n_{i})=\ell(w)+1$.
We proceed with the proof of these four claims by induction on $\ell(w)$.
We will repeatedly employ the identity 
\begin{equation}\label{eq: identity fund weights}
s_i^n(\varpi_j^n)=\varpi_j^n-\delta_{ij}\alpha_i^n.
\end{equation}

\begin{enumerate}[(a)]
\item We have that $ws^n_{i}(\alpha_{j}^n)=w(\alpha_{j}^n)-(\alpha_{j}^n,\alpha_{i}^n)\, w(\alpha_{i}^n)$, and Claim \eqref{claim1} follows by induction.
\item By definition \eqref{eq: def T phi}, we have the identity $T(s_i^n(\alpha_j^n)) = s_{i-u}^d(\alpha_{j-u}^d) = \varphi(s_{i}^d)(T(\alpha_j^n))$.
Claim \eqref{claim2} follows by noticing that
\begin{equation}
T(ws^n_{i}(\alpha_{j}^n))=\varphi(w)(T(s^n_{i}(\alpha_{j}^n)))=\varphi(ws_{i}^d)(T(\alpha_{j}^n)).
\end{equation}
\item We have that $w(\varpi_{j}^n) -\varpi_{j}^n\in  E^n(\Delta^n_J)$ by induction and $w(\alpha_{i}^n) \in E^n(\Delta^n_J)$ by \eqref{claim1}.
Therefore, using \eqref{eq: identity fund weights}, Claim \eqref{claim3} follows from the equality
\begin{equation}\label{eq: claim3}
ws^n_{i}(\varpi_{j}^n)-\varpi_{j}^n = w(\varpi_{j}^n) -\varpi_{j}^n -\delta_{ij}\, w(\alpha_{i}^n).
\end{equation}
\item We notice that
\begin{align*}
T(ws^n_{i}(\varpi_{j}^n)-\varpi_{j}^n)
&= T(w(\varpi_{j}^n) -\varpi_{j}^n -\delta_{ij}\, w(\alpha_{i}^n))
& \ \mbox{(\Cref{eq: claim3})} \\
&= T(w(\varpi_{j}^n) -\varpi_{j}^n) -\delta_{ij}\, T(w(\alpha_{i}^n)) &\ \mbox{(Claims \eqref{claim1}, \eqref{claim3})} \\
&= \varphi(w)(\varpi_{j-u}^d) - \varpi_{j-u}^d - \delta_{ij}\varphi(w)(\alpha_{i-u}^d)
& \ \mbox{(Induction, Claim \eqref{claim2})} \\
&= \varphi(w)(s_{j-i}^d(\varpi_{j-u}^d))-\varpi_{j-u}^d
& \ \mbox{(\Cref{eq: identity fund weights})} \\
&= \varphi(ws^n_{i})(U_J^n(\varpi_{j}^n))-U_J^n(\varpi_{j}^n),
\end{align*}
which is what we wanted to prove.
This concludes the proof of the four claims.
\end{enumerate}
Note that Claim \eqref{claim4} implies that applying $T$ to the set $\mathrm{Conv}(W_J^n\cdot\pi_J^n(\lambda))-\pi_J^n(\lambda)$, gives $\mathrm{Conv}(W^d\cdot U_J^n(\pi_J^n(\lambda)))- U_J^n(\pi_J^n(\lambda))$.
Since $T$ is an isometry, it follows that both sets have the same $d$-dimensional volume.
By \Cref{def: vols V_J} and the translation invariance of $\mathrm{Vol}_d$, Part \ref{lem: vol shift} follows.
\end{proof}

As a consequence we can see the volume $\mathcal{V}_n$ as a polynomial, as follows.
Let $i\in\{1,\ldots,n\}$ and consider the polynomial in $\bbR[x_1,x_2,\ldots,x_n]$ given by
\begin{equation}
p_i^n=(\varpi_1^n,\varpi_i^n)x_1+\cdots+(\varpi_n^n,\varpi_i^n)x_n.
\end{equation}
Let $J_i=\{s_1^n,\ldots,\widehat{s_i^n},\ldots,s_n^n\}$.
Then, for every $x_1,\ldots,x_n \geq0$, formula \eqref{eq: piramidal} gives 
\begin{equation}\label{eq: piramidal reescrita}
\mathcal{V}_n(x_1\varpi_1^n+\cdots +x_n\varpi_n^n)=\sum_{i=1}^nr_i \,p_i^n\,\mathcal{V}_{J_i}(x_1\varpi_1^n+\cdots +x_n\varpi_n^n),
\end{equation}
for some $r_i\in \bbR$.
Note that the connected components of $J_i$ are $\{s_1^n,\ldots,s_{i-1}^n\}$ and $\{s_{i+1}^n,\ldots,s_{n}^n\}$.
By Parts \ref{lem: vol cc}, \ref{lem: vol shift} of \Cref{lem: volgralfacts}, we get $\mathcal{V}_{J_i}$ in terms of $\mathcal{V}_{i-1}$ and $\mathcal{V}_{n-i}$, respectively (where we set $\mathcal{V}_0=1$).
Since $p_i^n$ is a homogeneous polynomial of degree $1$, we conclude by induction that $\mathcal{V}_n\in \bbR[x_1,x_2,\ldots,x_n]$ is homogeneous of degree $n$.
Hereafter, we work in the polynomial ring $\bbR[x_1,x_2,\ldots]$.

\begin{notation}
Let $d\in\bbZ_{ > 0}$.
From now on, $\mathcal{V}_d$ denotes the corresponding polynomial in the variables $x_1,\ldots,x_d$, as above.
That is, it is the polynomial defined by
\begin{equation*}
\mathcal{V}_d(x_1,x_2,\ldots)=\mathrm{Vol}_d\left(\mathsf{P}_d\left(x_1\varpi_1^d+\cdots +x_d\varpi_d^d\right)\right), \qquad\forall\, x_1,\ldots,x_d\in\bbR_{\geq0}.
\end{equation*}
We set $\mathcal{V}_0=1$.
\end{notation}

For a polynomial $p$ and a non-negative integer $u$, let $p[u]$ denote the shift of $p$ by $u$
\begin{equation*}
p[u](x_1,x_2,\ldots)=p(x_{1+u},x_{2+u},\ldots).
\end{equation*}

\begin{definition}\label{def: gamma pols}
For $d,i\in\bbZ_{ > 0}$ with $i\leq d$, let $\Gamma'_{d,i}$ be the homogeneous polynomial of degree $1$ defined by
\begin{equation}\label{eq: gamma pols}
\Gamma'_{d,i}(x_1,x_2,\ldots)=\dfrac{1}{d\sqrt{c_{d,i,i}}}\binom{d+1}{i}\sum_{j=1}^dc_{d,i,j}x_i,
\end{equation}
where $c_{d,i,j}=(\mathsf{A}_d^{-1})_{ij}$ (recall \Cref{eq: inversa cartan}).
We refer to the shifts of these polynomials as $\Gamma$-polynomials.
\end{definition}

\begin{remark}
\Cref{def: gamma pols} and the one given in the Introduction \eqref{eq: intro gamma pols} differ by a scalar (hence the notation $\Gamma')$.
\Cref{thm: vol as dyck} and \Cref{cor: vol gamma normalized} illustrate how each definition impacts the formula of the volume of the permutohedron.
\end{remark}

The following proposition rewrites formula \eqref{eq: piramidal} as a polynomial recursive formula.

\begin{prop}\label{prop: recurrence}
For every $n\in\bbZ_{\geq0}$, the following equality holds
\begin{equation}\label{eq: piramidal gamma}
\mathcal{V}_n=\sum_{i=1}^n \Gamma'_{n,i}\mathcal{V}_{i-1}\mathcal{V}_{n-i}[i].
\end{equation}
\end{prop}

\begin{proof}
Recall \Cref{eq: piramidal reescrita}.
Note that $W_{J_i}^n\cong W^{i-1}\times W^{n-i}$.
By \Cref{lem: piramidal}, we have
\begin{equation}
r_i=\dfrac{\left[W^n:W^n_{J_i}\right]}{n\sqrt{c_{n,i,i}}}= \dfrac{1}{n\sqrt{c_{n,i,i}}}\binom{n+1}{i},
\end{equation}
since $|W^d|=(d+1)!$.
It follows that
\begin{equation}
\mathcal{V}_n = \sum_{i=1}^n r_i \, p_i^n \, \mathcal{V}_{J_i} = \sum_{i=1}^n \Gamma'_{n,i}\mathcal{V}_{i-1}\mathcal{V}_{n-i}[i].
\end{equation}
\end{proof}

\subsection{Step 2}
From the recurrence \eqref{eq: piramidal gamma}, it is clear that $\mathcal{V}_n$ is a sum of $C_n$ terms, with $C_n$ the $n$-th Catalan number, where each term is a product of $\Gamma$-polynomials.
However it is not clear which $\Gamma$-polynomials appear.
We now precisely answer this question in terms of Dyck paths.
An $n$-Dyck path is a lattice path in $\mathbb{Z}^{2}$ from $(0,0)$ to 
$(n,n)$ with $n$ north steps (along the vector $(0,1)$) and $n$ east-steps (along the vector $(1,0)$), that lies above (but may touch) the diagonal $y=x$.
Let $\mathcal{D}_n$ be the collection of all $n$-Dyck paths.
For $D\in \mathcal{D}_n$, we denote a north step $N$ of $D$ as $N\in D$.

\begin{definition}\label{def: Dyck path poly gamma}
As in the Introduction, to each $D\in\mathcal{D}_n$ we associate a polynomial $\Gamma'_D$ as follows.
Let $D\in\mathcal{D}_n$ and $N\in D$.
Let $P_N=(u,u')$ be the initial point of $N$.
Let $D_N'$ be the sub-path of $D$ from $P_N$ to $(n,n)$, and let $D_N$ be translation of $D_N'$ by $-P_N$.
Let $K$ be the set of integers $1\leq k \leq n-u$ such that $D_N$ passes through $(k,k)$ and the sub-path of $D_N$ from $(0,0)$ to $(k,k)$ is a $k$-Dyck path.
We stress that $K$ is non-empty.
Let $d=\max(K)$ and $i=\min(K)$.
We define $\Gamma'_N= \Gamma'_{d,i}[u]$.
We have associated to each $N\in D$ a $\Gamma$-polynomial.
We define $\Gamma'_D$ as the  homogeneous polynomial of degree $n$ given by
\begin{equation*}
    \Gamma'_D = \prod_{N\in D} \Gamma'_{N}.
\end{equation*}
\end{definition}

\begin{example}
Let us illustrate \Cref{def: Dyck path poly gamma} with an example.
Let $D\in\mathcal{D}_7$ be the Dyck path in \Cref{fig: gammaN D7}, and let $N\in D$ be the north step corresponding to the thick segment.
The initial point of $N$ is at $(1,2)$.
For this step, we have $d=3$, $i=1$ and $u=1$.
Consequently, $\Gamma'_{N} = \Gamma'_{3,1}[1]$.
We leave to the reader to check that 
\begin{equation*}
\Gamma'_D= \Gamma'_{ 7,1 }[0]\,\Gamma'_{6 , 4}[1]\,\Gamma'_{ 3, 1}[1]\,\Gamma'_{ 2, 1}[2]\,\Gamma'_{ 1, 1}[3]\,\Gamma'_{2 , 2}[5]\,\Gamma'_{ 1,1}[5].
\end{equation*}
\begin{figure}[H]
\centering
\includegraphics[width=4cm]{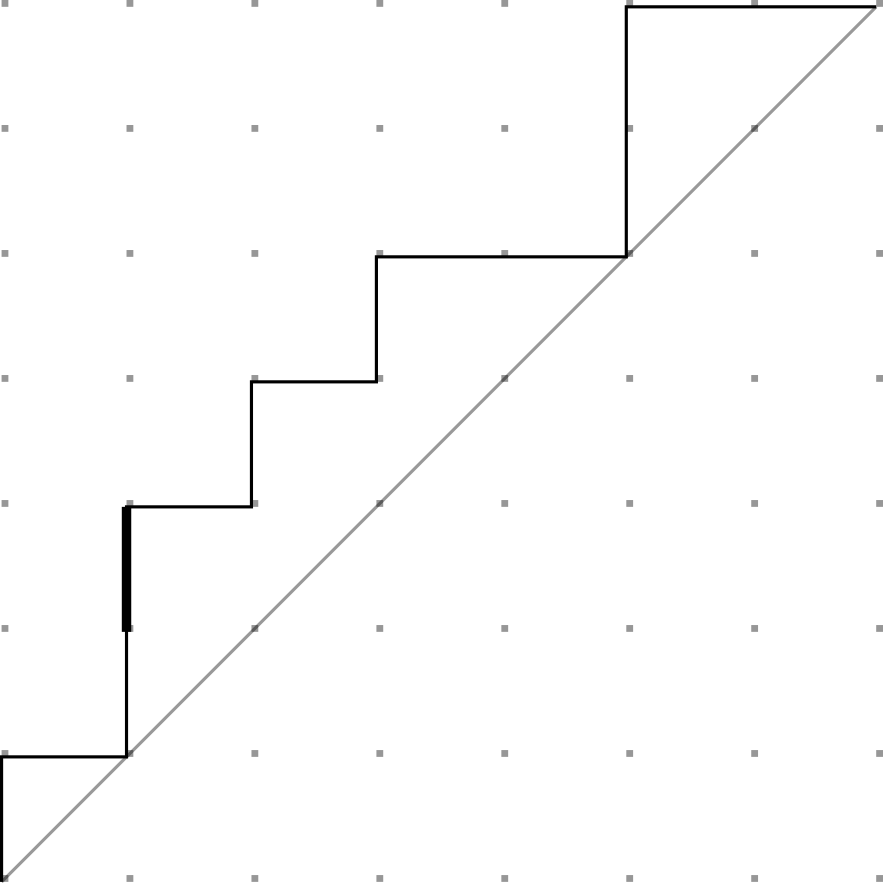}
\caption{A $7$-Dyck path and a north step.}
\label{fig: gammaN D7}
\end{figure}
\end{example}

\begin{theorem}\label{thm: vol as dyck}
For every $n\in\bbZ_{\geq 0}$, the following equality holds
\begin{equation}
\mathcal{V}_n=\sum_{D\in\mathcal{D}_n}\Gamma'_D.
\end{equation}
\end{theorem}

\begin{proof}
We proceed by induction on $n$.
If $n=0$ the result is trivial.
So we fix $n>0$ and assume the result holds for $m<n$.
For $1\leq k \leq n$ we define $\mathcal{D}_n^k$ to be the set of all $D\in \mathcal{D}_n$ such that $D$ passes through $(k,k)$ but $D$ does not pass through $(k',k')$ for $1\leq k'<k$.
To each $D\in\mathcal{D}_n^k$ we associate two Dyck paths $D_b\in\mathcal{D}_{k-1}$ and $D_t\in\mathcal{D}_{n-k}$, as follows:
\begin{itemize}
  \item $D_b$ is the path obtained from the sub-path of $D$ from $(0,1)$ to $(k-1,k)$ by decreasing the $y$ coordinates by $1$.
  \item $D_t$ is the path obtained from the sub-path of $D$ from $(k,k)$ to $(n,n)$ by decreasing the $x$ and $y$ coordinates by $k$.
\end{itemize}
Note that for $D\in \mathcal{D}_n^k$ we have
\begin{equation} \label{eq: gamma decomp in three}
 \Gamma'_{D} = \Gamma'_{n,k}\, \Gamma'_{D_b}\, \Gamma'_{D_t}[k], 
\end{equation} 
where the polynomial $\Gamma'_{n,k}$ comes from the north step from $(0,0)$ to $(0,1)$ eliminated in the construction of $D_b$, and the shift in $\Gamma'_{D_t}[k]$ comes from the decreasing by $k$ in the coordinates in the construction of $D_t$. 

On the other hand, the map $f_k:\mathcal{D}_n^k \rightarrow \mathcal{D}_{k-1} \times \mathcal{D}_{n-k} $ given by $f_k(D) = (D_b,D_t)$ is a bijection.
We get
\begin{align*}
\sum_{D\in\mathcal{D}_n}\Gamma'_D
&=\sum_{k=1}^n \sum_{D\in\mathcal{D}_n^k}\Gamma'_D & \\
&=\sum_{k=1}^n \sum_{D\in\mathcal{D}_n^k} \Gamma'_{n,k}\,\Gamma'_{D_b}\, \Gamma'_{D_t}[k]
& \ \ (\mbox{\Cref{eq: gamma decomp in three}}) \\
&=\sum_{k=1}^n \Gamma'_{n,k} \left(\sum_{D\in\mathcal{D}_{k-1}} \Gamma'_{D} \right) \left(\sum_{D\in\mathcal{D}_{n-k}} \Gamma'_{D} [k] \right)
& \ \ (f_k \mbox{ is a bijection})  \\
&=\sum_{k=1}^n \Gamma'_{n,k} \mathcal{V}_{k-1}\mathcal{V}_{n-k}[k]
& \ \ (\mbox{Induction})  \\
&=\mathcal{V}_{n}.
& \ \ (\mbox{\Cref{prop: recurrence}})
\end{align*}
\end{proof}

By our volume conventions (see \Cref{sec: vol conventions}), dividing $\mathcal{V}_n$ by $\sqrt{n+1}$, gives the volume of the permutohedron normalized so that the volume of the parallelepiped spanned by the simple roots has volume $1$, which is the normalization used in \cite{postnikov2009permutohedra} by Postnikov.
Dividing by this scalar can be realized in another way, as the following corollary explains.
Let $\Gamma_{d,i}=\sqrt{c_{d,i,i}}\,\Gamma'_{d,i}$.
The polynomial $\Gamma_{d,i}[u]$ is the definition \eqref{eq: intro gamma pols} used in the Introduction.
Let $N\in D\in\mathcal{D}_n$ and $\Gamma'_N=\Gamma'_{d,i}[u]$ as in \Cref{def: Dyck path poly gamma}.
We have
\begin{equation*}
\Gamma_N=\sqrt{c_N} \, \Gamma'_N
\qquad \mbox{and} \qquad
\Gamma_D = \sqrt{c_D}\, \Gamma'_D,
\end{equation*}
where $c_N=c_{d,i,i}$ and $c_D=\prod_{N\in D}c_N$.

\begin{cor}\label{cor: vol gamma normalized}
Let $n\in\bbZ_{\geq0}$.
For all $D\in\mathcal{D}_n$, we have $c_D=\dfrac{1}{n+1}$.
Consequently,
\begin{equation}\label{eq: gamma postnikov}
\mathcal{V}_n=\sqrt{n+1}\,\sum_{D\in\mathcal{D}_n}\Gamma_D.
\end{equation}
\end{cor}

\begin{proof}
We proceed by induction.
The case $n=0$ is clear, so we let $n>0$.
Let $D\in\mathcal{D}_n^k$ and let $N\in D$ with initial point $P_N=(0,0)$.
Note that $c_N=c_{n,k,k}=\dfrac{k}{n+1}(n+1-k)$, by \eqref{eq: inversa cartan}.
Using \eqref{eq: gamma decomp in three} and our inductive hypothesis, we get
\begin{equation*}
c_D = c_N\,c_{D_b}\,c_{D_t} = \dfrac{c_{n,k,k}}{k(n-k+1)} = \dfrac{1}{n+1},
\end{equation*}
which is what we wanted to prove.
\Cref{eq: gamma postnikov} follows immediately from \Cref{thm: vol as dyck}.
\end{proof}

\begin{remark}
Postnikov's formula \cite[Theorem 17.1]{postnikov2009permutohedra} is precisely $\dfrac{\mathcal{V}_n}{\sqrt{n+1}}=\sum_D\Gamma'_D$.
This equivalence can be derived using the following well-known bijection $\psi$ between $n$-Dyck paths and rooted planar binary trees with $n+1$ leaves.
For $D\in\mathcal{D}_n$, we inductively define $\psi(D)$ as the tree whose left and right subtrees (from the root) are $\psi(D_b)$ and $\psi(D_t)$, respectively.
\end{remark}

\subsection*{Acknowledgements}
The author would like to thank F. Castillo, N. Libedinsky and D. Juteau for their helpful comments.
A special thanks to D. Plaza for the insightful discussions.

\bibliography{bibliography}
\bibliographystyle{plain}

\end{document}